\documentclass[11pt,reqno]{amsart}
\usepackage{amssymb,latexsym}
\usepackage{graphicx}
\usepackage{fancyhdr,amssymb}
\usepackage{url}		%does nice formatting of URLs

\theoremstyle{plain}
\numberwithin{equation}{section}
\newtheorem{thm}{Theorem}[section]
\newtheorem{theorem}[thm]{Theorem}

\begin{document}
\fancyhead{}
\renewcommand{\headrulewidth}{0pt}
\fancyfoot{}
\fancyfoot[LE,RO]{\medskip \thepage}
%\fancyfoot[LO]{\medskip MONTH YEAR}
%\fancyfoot[RE]{\medskip VOLUME , NUMBER }

\setcounter{page}{1}

\title[On Squares in Lucas Sequences]{On Squares in Lucas Sequences}
\author{Francesca Balestrieri}
\address{Mathematical Institute\\
                University of Oxford\\
                UK}
\email{balestrieri@maths.ox.ac.uk}

\begin{abstract} In this short paper,  we prove, by only using elementary tools, general cases when $U_n(P,Q) \neq \square$, where $U_n(P,Q)$ is the Lucas sequence of the first type. \end{abstract}

\maketitle

\section{Introduction}
In this paper, we are interested in the \emph{Lucas sequence of the first type}, defined by
\begin{equation}\label{recrel}
U_n(P,Q) =P \cdot U_{n-1}(P,Q) - Q \cdot U_{n-2}(P,Q),
\end{equation}
where $P,Q$ are fixed integers,  with $U_0(P,Q) =0$, and  $U_1(P,Q)) =1$; in particular, we want to investigate for which values of $n$, $P$, and $Q$ we have that $U_n(P,Q)$ is not a square.
There are some results in the literature concerning squares in Lucas sequences.  For example, in \cite{ribenboim}, Ribenboim and McDaniel have proven that, provided that $P^2-4Q>0$ and $P,Q$ are odd and relatively prime, the only values of $n$ such that $U_n(P,Q) = \square$ are $n =0, 1, 2, 3, 6$ or $12 $ (their proof was extended in \cite{bremtza} by Bremner and Tzanakis, who, in the case for $n = 12$, removed the conditions $P^2-4Q>0$ and $P,Q$ odd, by using techniques from arithmetic geometry).
The symbol `$\square$' denotes a non-zero square.
The results we are going to prove in this paper do not require that $P^2 - 4Q >0$, at the cost of introducing some other mild constraints on $n, P,Q$.\\

\textbf{Remark 1.1. } Methods in arithmetic geometry can help us computing all the values of $P$ and $Q$ such that $U_n(P,Q) = \square$, for a fixed $n$; see \cite{bremtza2} for many examples.
The downside of this approach is that, for $n \geq 15$, the hyperelliptic curves $y^2 = f(x)$ obtained by setting $x = \frac{Q}{P^2}$ in the expressions for $U_n(P,Q)$ have genus $g \geq 3$, meaning that they are not easily tractable, at least using the current methods and tools available (most of MAGMA routines, for example, apply only to curves of genus $g \leq 2$). On the other hand, for $n \geq 11$, the hyperelliptic curves obtained in the way described above have genus $g \geq 2$; by Faltings' theorem \cite{falt}, it follows that there are only finitely many values of coprime $P$ and $Q$ such that $U_n(P,Q) = \square$, for $n \geq 11$.

\section{Preliminaries}

There is an interesting relation between the coefficients of the $P^\alpha Q^\beta$ terms in $U_n(P,Q)$ and the Pascal triangle: these coefficients are, up to a sign, the elements in the diagonals of the Pascal triangle. This motivates the following (known, see \cite{lucas}) closed formula.

\begin{theorem} For $n \geq 1$, we have
\begin{equation}\label{formula}
U_n(P,Q) = \sum_{r=0}^{\lfloor \frac{n - 1}{2}\rfloor} (-1)^r P^{n-1 - 2r} Q^r \binom {n-1-r} {r}. 
\end{equation}
\end{theorem}

\begin{proof} This is an easy induction. We write $U_n$ for $U_n(P,Q)$, in order to ease notation. The base cases for $n =1,2$ hold.

For $n \geq 3$, we need to distinguish two cases.
\begin{enumerate}
\item \textbf{Case $ \lfloor \frac{n - 2}{2}\rfloor = \lfloor \frac{n - 3}{2}\rfloor + 1$.} Then $n$ is even, and $ \lfloor \frac{n - 2}{2}\rfloor = \lfloor \frac{n - 1}{2}\rfloor$. We have
\[
\begin{array}{rl}
U_n = & P \cdot U_{n-1} - Q \cdot U_{n-2}\\
  =&  \sum_{r=0}^{\lfloor \frac{n - 2}{2}\rfloor} (-1)^r P^{n-1 - 2r} Q^r \binom {n-2-r} {r} + \sum_{r=0}^{\lfloor \frac{n - 3}{2}\rfloor} (-1)^{r+1} P^{n-3 - 2r} Q^{r+1} \binom {n-3-r} {r}\\
=&  \sum_{r=0}^{\lfloor \frac{n - 2}{2}\rfloor} (-1)^r P^{n-1 - 2r} Q^r \binom {n-2-r} {r} + \sum_{k=1}^{\lfloor \frac{n - 3}{2}\rfloor+1} (-1)^k P^{n-3 - 2(k-1)} Q^{k} \binom {n-3-k+1} {k-1}\\
=&  \sum_{r=0}^{\lfloor \frac{n - 2}{2}\rfloor} (-1)^r P^{n-1 - 2r} Q^r \binom {n-2-r} {r} + \sum_{k=1}^{\lfloor \frac{n - 2}{2}\rfloor} (-1)^k P^{n-1 - 2k} Q^{k} \binom {n-2-k} {k-1}\\
=&  \sum_{r=1}^{\lfloor \frac{n - 2}{2}\rfloor} (-1)^r P^{n-1 - 2r} Q^r (\binom {n-2-r} {r} + \binom{n-2-r}{r-1})
+  P^{n-1} \\
=&  \sum_{r=1}^{\lfloor \frac{n - 2}{2}\rfloor} (-1)^r P^{n-1 - 2r} Q^r \binom {n-1-r} {r} 
+  P^{n-1} \\
=&  \sum_{r=0}^{\lfloor \frac{n - 1}{2}\rfloor} (-1)^r P^{n-1 - 2r} Q^r \binom {n-1-r} {r}, 
\end{array}
\]
where we have used the  standard binomial identity $\binom {k} {r} + \binom{k}{r-1}= \binom {k+1} {r} $.

\item \textbf{Case $ \lfloor \frac{n - 2}{2}\rfloor = \lfloor \frac{n - 3}{2}\rfloor$.} Then $n$ is odd, and $\lfloor \frac{n - 2}{2}\rfloor +1 = \lfloor \frac{n - 1}{2}\rfloor = \frac{n - 1}{2}$. We have
\[
\begin{array}{rl}
U_n = & P \cdot U_{n-1} - Q \cdot U_{n-2}\\
  =&  \sum_{r=0}^{\lfloor \frac{n - 2}{2}\rfloor} (-1)^r P^{n-1 - 2r} Q^r \binom {n-2-r} {r} + \sum_{r=0}^{\lfloor \frac{n - 3}{2}\rfloor} (-1)^{r+1} P^{n-3 - 2r} Q^{r+1} \binom {n-3-r} {r}\\
=&  \sum_{r=0}^{\lfloor \frac{n - 2}{2}\rfloor} (-1)^r P^{n-1 - 2r} Q^r \binom {n-2-r} {r} + \sum_{k=1}^{\lfloor \frac{n - 3}{2}\rfloor+1} (-1)^k P^{n-3 - 2(k-1)} Q^{k} \binom {n-3-k+1} {k-1}\\
=&  \sum_{r=0}^{\lfloor \frac{n - 2}{2}\rfloor} (-1)^r P^{n-1 - 2r} Q^r \binom {n-2-r} {r} + \sum_{k=1}^{\lfloor \frac{n - 2}{2}\rfloor+1} (-1)^k P^{n-1 - 2k} Q^{k} \binom {n-2-k} {k-1}\\
=&  \sum_{r=0}^{\lfloor \frac{n - 1}{2}\rfloor-1} (-1)^r P^{n-1 - 2r} Q^r \binom {n-2-r} {r} + \sum_{k=1}^{\lfloor \frac{n - 1}{2}\rfloor} (-1)^k P^{n-1 - 2k} Q^{k} \binom {n-2-k} {k-1}\\
=&  \sum_{r=1}^{\lfloor \frac{n - 1}{2}\rfloor-1} (-1)^r P^{n-1 - 2r} Q^r (\binom {n-2-r} {r} + \binom{n-2-r}{r-1})
+  P^{n-1} \\
& + (-1)^{\lfloor \frac{n-1}{2} \rfloor} P^{n-1-2\lfloor \frac{n-1}{2} \rfloor}Q^{\lfloor \frac{n-1}{2} \rfloor} \binom{n-2-\lfloor \frac{n-1}{2} \rfloor}{\lfloor \frac{n-1}{2} \rfloor-1} \\
=&  \sum_{r=1}^{\lfloor \frac{n - 1}{2}\rfloor-1} (-1)^r P^{n-1 - 2r} Q^r \binom {n-1-r} {r} 
+  P^{n-1}+ (-1)^{\lfloor \frac{n-1}{2} \rfloor} P^{n-1-2\lfloor \frac{n-1}{2} \rfloor}Q^{\lfloor \frac{n-1}{2} \rfloor} \binom{n-1-\lfloor \frac{n-1}{2} \rfloor}{\lfloor \frac{n-1}{2} \rfloor} \\
=&  \sum_{r=0}^{\lfloor \frac{n - 1}{2}\rfloor} (-1)^r P^{n-1 - 2r} Q^r \binom {n-1-r} {r},
\end{array}
\]
where the penultimate equality follows as $\binom{n-2-\lfloor \frac{n-1}{2} \rfloor}{\lfloor \frac{n-1}{2} \rfloor-1} = \binom{n-2- \frac{n-1}{2}} {\frac{n-1}{2}-1} = \binom{\frac{n-3}{2}}{\frac{n-3}{2}}=1$ and  \\ $\binom{n-1-\lfloor \frac{n-1}{2} \rfloor}{\lfloor \frac{n-1}{2} \rfloor} =\binom{\frac{n-1}{2}}{\frac{n-1}{2}}=1$.
\end{enumerate}
By induction, the result follows.\end{proof}

\section{Our Results}
We will use the elementary fact that any perfect square (in $\mathbb{N}$) is congruent to 0 or 1 modulo 4.
Suppose that $P, Q, n$ are odd. Note that since $P$ is odd, $P^2 \equiv 1$ (mod 4). Consider the following examples, which we will try to generalise:
\begin{enumerate}
\item $n = 3$: Suppose $Q \equiv 3$ (mod 4). Then $U_{3}(P,Q) = P^2 - Q \equiv 1 - 3 \equiv 2$ (mod 4), so $U_3(P,Q) \neq \square$.
\item $n = 5$:  Suppose $Q \equiv 1$ (mod 4). Then $U_{5}(P,Q) =P^4 - 3P^2Q + Q^2 \equiv 1-3Q +1 \equiv 2 - 3Q \equiv 3$ (mod 4), so $U_5(P,Q) \neq \square$.
\item $n = 7$:  We have that $U_{7}(P,Q) = P^6 - 5 P^4 Q + 6 P^2 Q^2 - Q^3 \equiv 1 - Q + 2 - Q \equiv 3 - 2Q \equiv 1$ (mod 4), so we cannot say anything more.
\item $n = 9$:  Suppose $Q \equiv 3$ (mod 4). Then  $U_{9}(P,Q) =P^8 - 7 P^6 Q + 15 P^4 Q^2 - 10 P^2 Q^3 + Q^4 \equiv 1-7Q+15-10Q+1 \equiv 3Q + 1 \equiv -3+1 \equiv 2$ (mod 4). So $U_9(P,Q) \neq \square$.
\item $n = 11$:  Suppose $Q \equiv 1$ (mod 4). Then  $U_{11}(P,Q) =P^{10} - 9 P^8 Q + 28 P^6 Q^2 - 35 P^4 Q^3 + 15 P^2 Q^4 - Q^5 \equiv 1 +3Q+Q+3-Q \equiv 3Q \equiv 3$ (mod 4). So $U_{11}(P,Q) \neq \square$.
\item $n = 13$:  We have that  $U_{13}(P,Q) =P^{12} - 11 P^{10} Q + 45 P^8 Q^2 - 84 P^6 Q^3 + 70 P^4 Q^4 - 21 P^2 Q^5 + Q^6 \equiv 1 + Q + 1 + 2 - Q + 1 \equiv 1 $ (mod 4). So we cannot say anything more.
\item $n = 15$:  Suppose $Q \equiv 3$ (mod 4). Then  $U_{15}(P,Q) = P^{14} - 13 P^{12} Q + 66 P^{10} Q^2 - 165 P^8 Q^3 + 210 P^6 Q^4 - 126 P^4 Q^5 + 28 P^2 Q^6 - Q^7 \equiv 1 +3Q +2 +3Q+2 + 2Q+3Q \equiv 1 + 3Q \equiv 2$ (mod 4). So $U_{15}(P,Q) \neq \square$.
\item $n = 17$:  Suppose $Q \equiv 1$ (mod 4). Then  $U_{17}(P,Q) =P^{16} - 15 P^{14} Q + 91 P^{12} Q^2 - 286 P^{10} Q^3 + 495 P^8 Q^4 - 462 P^6 Q^5 + 210 P^4 Q^6 - 36 P^2 Q^7 + Q^8 \equiv  1 +  Q + 3 +2Q + 3 + 2Q + 2 + 1 \equiv  Q + 2 \equiv 3 $ (mod 4). So $U_{17}(P,Q) \neq \square$.
\end{enumerate} 

A pattern clearly emerges, and motivates the following. 

\begin{theorem} For $P,Q, n$ odd, we have
\begin{equation}
U_n(P,Q) \neq \square \textrm{ if }
\begin{cases}
n \equiv 3 \textrm{ (mod 6) and } Q \equiv 3 \textrm{ (mod 4), }\\ 
n \equiv 5 \textrm{ (mod 6) and } Q \equiv 1 \textrm{ (mod 4).}\\ 
\end{cases}
\end{equation}
\end{theorem}

\begin{proof} Let $n \equiv 3 \textrm{ (mod 6)}$ and $Q \equiv -1$ (mod 4). From \eqref{formula}, we get
\begin{equation}\label{case1}
U_n(P,Q) \equiv \sum_{r=0}^{\lfloor \frac{n - 1}{2}\rfloor} \binom {n-1-r} {r} \textrm{ (mod 4)},
\end{equation}
since the powers of $P$ in the sum are always even (since $n$ is odd), meaning  (since $P$ is odd) that these powers are congruent to 1 modulo 4. But \eqref{case1} implies that $U_n(P,Q) \equiv U_n(1,-1) =  F_n$ (mod 4), where $F_n$ is the $n$-th Fibonacci number (starting from $F_1 = 1$, $F_2=1$, $F_3 = 2$, and so on). Then the result follows immediately: consider
\[
\begin{array} {l|cccccccccc}
n & 1 & 2& \textbf{3}& 4& 5& 6& 7& 8& \textbf{9} & ...\\
\hline
F_n& 1 & 1 & \textbf{2} & 3&5&8&13&21&\textbf{34} & ...\\
F_n \textrm{ (mod 4)} & 1 & 1 & \textbf{2} & 3&1&0&1&1&\textbf{2} & ... .\\
\end{array}\]
As we can see from the above table, reduction modulo 4 of the Fibonacci numbers is cyclic, with period $1,1,2,3,1,0$ of length 6. This means that $U_n(P,Q) \equiv F_n \equiv 2$ (mod 4) precisely when $n \equiv 3$ (mod 6), as required.

Now let $n \equiv 5 \textrm{ (mod 6)}$ and $Q \equiv 1$ (mod 4). From \eqref{formula} again, we can deduce that
\begin{equation}\label{case2}
U_n(P,Q) \equiv \sum_{r=0}^{\lfloor \frac{n - 1}{2}\rfloor} (-1)^r \binom {n-1-r} {r} \textrm{ (mod 4)},
\end{equation}
where the same remarks on the powers of $P$ as in the previous case hold. But \eqref{case2} implies that $U_n(P,Q) \equiv U_n(1,1)$ (mod 4), where $U(1,1)_n$ is the sequence $U_n(1,1) = U_{n-1}(1,1) - U_{n-2}(1,1)$, with  $U_{0}(1,1) = 0$ and $U_{0}(1,1) = 1$. Then, by considering the table
\[
\begin{array} {l|cccccccccccc}
n & 1 & 2& 3& 4& \textbf{5}& 6& 7& 8& 9 & 10 & \textbf{11} &...\\
\hline
U_n(1,1)& 1 & 1 & 0 & -1 & \textbf{$-1$} & 0 & 1&1 & 0 & -1 &  \textbf{$-1$} & ...\\
U_n(1,1) \textrm{ (mod 4)} & 1 & 1 & 0 & 3 & \textbf{3} & 0 & 1&1 & 0 &3 & \textbf{3} & ...\\
\end{array}\]
the result follows, as reduction modulo 4 of the $U_n(1,1)$ is cyclic with period $1,1,0,3,3,0$ of length 6. This means that $U_n(P,Q) \equiv U_n(1,1) \equiv 3$ (mod 4) when $n \equiv 5$ (mod 6), as required.\end{proof}

In a similar spirit, we prove the following.

\begin{theorem} For $P, n$ odd, $n \geq 3$, and $Q \equiv 2$ (mod 4), we have $U_n(P,Q) \neq \square$.
\end{theorem}

\begin{proof}  From \eqref{formula}, we get
\begin{equation}\label{case3}
U_n(P,Q) \equiv \sum_{r=0}^{\lfloor \frac{n - 1}{2}\rfloor} (-1)^r 2^r\binom {n-1-r} {r} \equiv 1 + (-1)\cdot 2 \cdot \binom{n-2}{1} \equiv 1 + 2n \equiv 3 \textrm{ (mod 4)},
\end{equation}

since the powers of $P$ in the sum are always even (since $n$ is odd), meaning  (since $P$ is odd) that these powers are congruent to 1 modulo 4, since $n\geq 3$ implies that $r \geq 1$, and since for $r \geq 2$ the powers $Q^r \equiv 2^r$ vanish modulo 4.
\end{proof}

\begin{theorem}For $P$ odd such that $P \equiv -1$ (mod 4), $n$ even, and $Q \equiv 0$ (mod 4), we have $U_n(P,Q) \neq \square$.
\end{theorem}

\begin{proof}  From \eqref{formula}, we get
\begin{equation}\label{case4}
U_n(P,Q) \equiv \sum_{r=0}^{\lfloor \frac{n - 1}{2}\rfloor} (-1)^r (-1)^{n-1-2r}4^r\binom {n-1-r} {r} \equiv (-1)^{n-1} \equiv -1 \textrm{ (mod 4)},
\end{equation}
as $n-1$ is odd, and for $r \geq 1$, the powers $Q^r \equiv 4^r$ vanish modulo 4.
\end{proof}

\begin{theorem} For $P \equiv 2,3$  (mod 4)  and $n=2$, we have $U_2(P,Q) \neq \square$.
\end{theorem}

\begin{proof}  From \eqref{formula}, we get
\begin{equation}\label{case5}
U_2(P,Q) \equiv  P \equiv 2,3 \textrm{ (mod 4)},
\end{equation}
where we have used the fact that, for $n=2$, $\lfloor\frac{n-1}{2}\rfloor = 0$.
\end{proof}

\medskip

\noindent MSC2013: 11B39

\end{document}